\documentclass[runningheads]{llncs}
\usepackage{graphicx}
\usepackage[english]{babel}
\usepackage[utf8]{inputenc}
\usepackage{typearea}
\usepackage{latexsym}
\usepackage{lineno,hyperref}
\usepackage{amsmath,amssymb,mathtools}
\usepackage{algorithm}
\usepackage{algorithmic}
\usepackage{graphicx}
\usepackage[abs]{overpic}

\newcommand{\R}{\mathbb{R}}

\newcommand{\cc}{\tilde{c}}
\def\M{\mathcal{M}}

\newcommand{\Sx}{\mathcal{S}}
\newcommand{\Dx}{\mathcal{D}}

\newcommand{\df}{Diff^+}

\newcommand{\ga}{\gamma}

\newcommand{\vphi}{\varphi}

\newcommand*{\quark}{\setbox0\hbox{$x$}\hbox to\wd0{\hss$\cdot$\hss}}

\usepackage{xcolor}
\newcommand{\todo}[1]{\bgroup\color{red}#1\egroup}

\begin{document}
\title{Elastic Analysis of Augmented Curves and Constrained Surfaces}
%
%
\author{Esfandiar Nava-Yazdani\orcidID{0000-0003-4895-739X} 
}
\authorrunning{E. Nava-Yazdani}
%
\institute{Zuse Institute Berlin, Berlin, Germany 
\email{navayazdani@zib.de}\\
\url{https://www.zib.de/members/navayazdani}
}
\maketitle              
\begin{abstract}
The square root velocity transformation is crucial for efficiently employing the elastic approach in functional and shape data analysis of curves. We study fundamental geometric properties of curves under this transformation. Moreover, utilizing natural geometric constructions, we employ the approach for intrinsic comparison within several classes of surfaces and augmented curves, which arise in the real world applications such as tubes, ruled surfaces spherical strips, protein molecules and hurricane tracks.

\end{abstract}
\section{Introduction}
Metric comparison of curves is a core task in a wide range of application areas such as morphology, image and shape analysis, computer vision, action recognition and signal processing. Thereby, a Riemannian structure is highly desirable, since it naturally provides powerful tools, beneficial for such applications.     

In the recent years, the use of Riemannian metrics for the study of sequential data, such as shapes of curves, trajectories given as longitudinal data or time series, has rapidly grown. In elastic analysis of curves, one considers deformations caused from both bending and stretching. A Riemannian metric, which quantifies the amount of those deformations is called elastic (cf.~\cite{Mio2006OnSO,mio07plane}). Therein, in contrast to landmark-based approaches (cf.~\cite{Ken1999,nava2020geo,NavaYazdaniHegevonTycowicz2022}), one considers whole continuous curves instead of finite number of curve-points. Consequently, the underlying spaces are infinite dimensional and computational cost becomes a significant issue. The square root velocity (SRV) framework provides a convenient and numerically efficient approach for analysing curves via elastic metrics and has been widely used in the recent years (cf.~\cite{sri2010protein,huang2016riemannian,bbhm2018totgeo,bauer18relaxed} and the comprehensive work~\cite{sri2016functional}).

In many applications the curves are naturally manifold-valued. For instance, Lie groups such as the Euclidean motion group, or more generally, symmetric spaces including the Grassmannian and the Hadamard-Cartan manifold of positive definite matrices are widely used in modelling of real world applications. Extensions of SRV framework from euclidean to general manifold-valued data can be found in~\cite{le2017computing,zhang2018rate,homo,homo2018shape,homo2018comparing}.

Our contributions are the following. We expose for plane curves the behaviour of speed and curvature under the SRV transformation and geometric invariants. Moreover, we apply the elastic approach to augmented curves, determining certain classes of surfaces, tubes, ruled surfaces and spherical strips, as well as hurricane tracks considered with their intensities. We recall that with distance and geodesic at hand, significant ingredients of statistical analysis such as mean and principal geodesic components as well as approximation and modelling concepts such as splines can be computed.

This paper is organized as follows. Section \ref{sec:1}, presents the Riemannian setting and notations. Section \ref{sec:2} is devoted to applications. Therein, we consider time series, for which in addition to spatial data, auxiliary information give rise to augmented curves and some classes of surfaces generated by them. Thereby, we apply the elastic approach to both euclidean and spherical trajectories. Future prospects and concluding remarks are presented in \ref{sec:3}.

For the convenience of those readers primary interested in the applications, we mention that, advanced parts and details from differential geometry, presented in \ref{sec:1}, can be skipped. Thereby, the essential point is the use of a framework (SRV) for computation of shortest paths on the spaces of curves and their shapes. 
\section{Riemannian Framework}
\label{sec:1}
\subsection{Preliminaries}
For the background material on Riemannian geometry, we refer to \cite{doCarmo1992} and \cite{gallot}. Let $(M,g )$ be a finite dimensional Riemannian manifold and $\M$ the Fr{\'e}chet manifold of smooth immersed curves from $\Dx$ in $M$, where $\Dx$ denotes either the unit circle $S^1$ or the unit interval $I:=[0,1]$ for closed or open curves respectively. Moreover, we denote the group of orientation preserving diffeomorphisms on $\Dx$ by $\df$. The following reparametrization invariance is crucial for a Riemannian metric $G$ on $\M$: 
\[
G_{c\circ \vphi}(h\circ \vphi,k\circ\vphi)=G_c(h,k),
\]
for any $c\in \M$,\, $h,k\in T_c\M$ and $\vphi\in\df$. The above equivariance ensures that the induced distance function satisfies the following, which is often desirable in applications:
\[
d(c_0\circ\vphi,c_1\circ\vphi)=d(c_0,c_1),
\]
for any two curves $c_0$ and $c_1$ in $\M$. Similarly, denoting the isometry group of $M$ by $Isom(M)$ and the tangent map of $F\in Isom(M)$ by $TF$, the invariance
\[
G_{F\circ c}(TF\circ h,TF\circ k)=G_c(h,k),
\]
ensures that
\[
d(F\circ c_0,F\circ c_1)=d(c_0,c_1).
\]
With the above invariances, we can divide out the spaces $Isom(M)$ and $\df$, and consider the natural induced distance $d^S$ on the quotient space
\[
\Sx=\M/(\df\times Isom(M))
\]
given by
\begin{align*}
d^S([c_0],[c_1])&=inf\,\{d(c_0,f\circ c_1\circ\vphi):\, \vphi\in \df,\,f\in Isom(M)\}\\
&=inf\,\{d(f\circ c_0\circ\vphi,c_1):\, \vphi\in \df,\,f\in Isom(M)\}.
\end{align*}
In the context of shape analysis of curves, $\M$ and $\Sx$ are called the pre-shape and shape space, respectively. Note that the order of quotient operations does not matter, since the left action of $Isom(M)$ and the right action of $\df$ commute. $\M/\df$ is the space of unparametrized curves and its inherited distance reads 
\[
inf\,\{d(c_0,c_1\circ\vphi):\, \vphi\in \df\}.
\]
We remark that particular essential challenges are due to the fact that some basic concepts and results from finite dimensional differential geometry such as Hopf-Rinow theorem, do not carry over to the infinite dimensional case. Now, let $\nabla$ be the Levi-Civita connection of $M$ and denote the arc length parameter, speed and unit tangent of $c$ by $\theta$, $\omega$ and $T$ respectively. Thus, we have $\omega = |\dot{c}|$, $d\theta=\omega dt$ and $T=\frac{\dot{c}}{\omega}$, where dot stands for derivation with respect to the parameter $t$. 

Due to a remarkable result in ~\cite{michor2005vanishing} the geodesic distance induced by the simplest natural choice, the $L^2$-metric
\[
G_c^{L^2}(h,k)=\int_\Dx g_c( h,k) \,d\theta,
\]
always vanishes. Consequently, some stronger Sobolev metrics have been considered in several works including~ \cite{michor2007overview,bauer2013sobolev,bauer2014overview}. They are given by
\[
G_c(h,k)=\sum_{i=0}^n \int_\Dx a_ig_c( \nabla_T^ih, \nabla_T^ik)\, d\theta,
\]
with $a_1$ non-vanishing and all $a_i$ non-negative, distinguish the curves. We consider first order metrics with constant coefficients. We remark that the coefficients $a_i$ can be chosen such that the metric is scale invariant, which is a desired property for some applications in shape analysis. A family of certain weighted Sobolev-type metrics, the so-called elastic metrics, based on a decomposition of derivatives of the vector fields into normal and tangent components, has been introduced in ~\cite{Mio2006OnSO,mio07plane}:
\[
G_c^{a,b}(h,k)=\int_\Dx a g_c((\nabla_T h)^\top, (\nabla_T k)^\top) + b g_c((\nabla_T h)^\perp, (\nabla_T k)^\perp )\, d\theta,
\]
with $4b\geq a>0$. In this work, we use the square root velocity (SRV) framework, which allows for a convenient and computationally efficient elastic approach. The main tool in this framework is the square root velocity transformation, which for euclidean $M$ reads 
\[
q:\,c\mapsto\frac{\dot{c}}{\sqrt{|\dot{c}|}}.
\]
It isometrically maps curves modulo translations, with the metric $G^{1,1/4}$ to $\M$ with the flat $L^2$-metric given by 
\[
G^0(v,w)= \int_\Dx g( v(t),w(t)) dt.\] 
This metric is frequently called (cf. \cite{michor2007metric,Bauer2012ConstructingRI,bauer2020intrinsic}) flat, to emphasize its footpoint independence. Note that the elastic metric $G^{1,1}$ corresponds to the first order Sobolev metric with $a_0=0$ and $a_1=1$. We remark, that for plane curves, the work \cite{srv_extension} has extended the SRV transformation to general parameters $a,b>0$. For further reading on the SRV framework and applications in shape analysis, we refer to \cite{sri2010protein}, \cite{huang2016riemannian} (numerical aspects), the survey \cite{bauer2020intrinsic} and particularly, the comprehensive work \cite{sri2016functional}. 
\subsection{Plane Curves}
A natural question that arises is, how essential geometric characteristics of a curve behave under the SRV transformation. In the following, we provide an answer for speed and curvature in the case of plane curves. Let $M=\R^2$, $\tilde{c}:=q(c)$ and denote the curvature of $c$ by $\kappa$. Note that $\tilde{c}$ does not need to be an immersion.
\begin{proposition}
Denoting the speed of $\tilde{c}$ by $\tilde{\omega}$, we have
\begin{equation}\label{eq:omega}
    \tilde{\omega}=\sqrt{\frac{\dot{\omega}^2}{4\omega}+\omega^3\kappa ^2}.
\end{equation}
Moreover, $\tilde{c}$ is an immersion if and only if $\kappa$ and $\dot{\omega}$ have no common zeros. In this case,
\begin{equation}\label{eq:kappa}
    \tilde{\kappa}\tilde{\omega}=\kappa\omega + \dot{\varphi},
\end{equation}
where $\tilde{\kappa}$ denotes the curvature of $\tilde{c}$ and
\[
\varphi:=\arctan\left(\frac{2\omega^2\kappa}{\dot{\omega}}\right).
\]
\end{proposition}
\begin{proof}
Let $N$ denote the unit normal of $c$. With the shorthand notations $\alpha:=\sqrt{\omega}$ and $\beta:=\alpha^3 \kappa$, a straightforward application of the Frenet equations $\dot{T}=\omega\kappa N$ and $\dot{N}=-\omega\kappa T$, yields
\begin{align*}
\dot{\cc} &= \dot{\alpha} T+ \beta N,\\
\ddot{\cc}&= (\ddot{\alpha} - \frac{\beta^2}{\alpha})T+(\dot{\beta}+\frac{\dot{\alpha}\beta}{\alpha})N.
\end{align*}
Thus, we have 
\[
\tilde{\omega}=\sqrt{\dot{\alpha}^2+\beta^2},
\]
immediately implying \eqref{eq:omega}. Obviously, zeros of $\tilde{\omega}$ are common zeros of $\kappa$ and $\dot{\omega}$. Thus, $\tilde{c}$ is an immersion if and only if $\kappa$ and $\dot{\omega}$ have no common zeros. In this case, $\tilde{\kappa}$ and $\varphi=\arctan\left(\beta/\dot{\alpha}\right)=\arctan\left(\frac{2\omega^2\kappa}{\dot{\omega}}\right)$ are well-defined and 
\[
\tilde{\kappa}\tilde{\omega}^3=\tilde{\omega}^2\beta/\alpha +\dot{\alpha}\dot{\beta}-\ddot{\alpha}\beta,
\]
which immediately implies the curvature formula \eqref{eq:kappa}.
\end{proof}
Next, we apply the proposition to study some geometric quantities, which are invariant under the SRV transformation. For closed curves, integrating the curvature formula above over $\Dx=S^1$ (note that in this case, $\tilde{\omega}>0$ almost everywhere), we see that the SRV transformation preserves the total curvature and particularly the turning number. Moreover, $\kappa\omega$ is preserved if and only if $\kappa=a\frac{d}{dt}\left(\frac{1}{\omega}\right)$ with a constant $a$.

Clearly, with $\kappa$ and $\omega$ at hand, utilizing Frenet equations, we can compute $c$ up to rigid motions. The following explicit solution is an immediate application of the above proposition. In light of the above proposition, immersed curves, which are mapped to straight lines, can easily be determined as follows.
\begin{example}
Let $a,b,A$ be constants with $ab,\,A>0$, $\omega (t)=A/\sin^2(at+b)$ and $\kappa =a/\omega$. A straightforward computation, utilizing the curvature formula \eqref{eq:kappa}, implies $\tilde{\kappa}=0$.
\end{example}
\subsection{Curves in Homogeneous Spaces}
\label{subsec:1.3}
For the background material on Lie groups and homogeneous spaces, we refer to \cite{gallot}. The works~\cite{le2017computing,zhang2018rate} provide extensions of the SRV framework for euclidean curves to the case of general manifolds. The former has high computational cost, while the latter, transported SRV, depends on a reference point and also suffers from distortion or bias caused by holonomy effects. We use the natural extension to homogeneous spaces exposed in \cite{homo2018comparing,homo2018shape}. For reader's convenience, we sketch the core ingredients of the approach and refer to the mentioned works for details and some applications. 

Let $M$ be a homogeneous space, i.e., $M=H/K$, where $K$ is a closed Lie subgroup of a Lie Group $H$. Let $\|\cdot\|$ denote the induced norm by a 
left invariant metric on $H$, $L$ the tangent map of the left translation, and $Imm(\Dx, H)$ the space of immersed curves from $\Dx$ to $H$. The SRV transformation is given by $Q(\alpha)=(\alpha (0), q(\alpha))$, where 
\[
q(\alpha)=
\frac{L_{\alpha^{-1}}\dot{\alpha}}{\sqrt{ \|\dot{\alpha}\|}}
\]
Here, $\alpha^{-1}(t)$ denotes the inverse element of $\alpha (t)$ in $H$ and $\mathcal{H}$ the Lie algebra of $H$. The map $Q$ is a bijection from $Imm(\Dx, H)$ onto $H\times L^2(\Dx,\mathcal{H})$. Now, $\M$ can be equipped with the Riemannian metric given by the pullback of the product metric of $H\times L^2(\Dx,\mathcal{H})$ using the map $Q$ and horizontal lifting. Let $c_1$ and $c_2$ be immersed curves in $M$ with horizontal lifts $\alpha_1$ and $\alpha_2$ respectively. The induced distance on $\M$ reads
\[
d(c_1,c_2) =inf \left\{ \sqrt{ d_H^2(\alpha_1(0),\alpha_2(0)x)+\| q(\alpha_1)-Ad_{x^{-1}}(q(\alpha_2) \|_{L^2}^2 } :\: x\in K \right\}.
\]
\section{Applications}
\label{sec:2}
Frequently, besides spatiotemporal data, represented by a curve $\ga$ in a manifold $M$, there are additional or auxiliary information associated with the curve, thus with the same time-correspondence. These can jointly with $\ga$ be comprised and represented as a so-called augmented curve $\tilde{\ga}$ in a higher dimensional manifold $\tilde{M}$. In some applications, the curve $\tilde{\ga}$ uniquely determines a submanifold $N$ of $M$ via a natural construction. An important example is provided, when $\tilde{M}$ is a submanifold of the tangent bundle of $M$, where the auxiliary information is represented as a vector field along $\ga$ and the construction is given by the Riemannian exponential map. Significant special cases occur, when $M$ is $\R^3$ or the unit two-sphere $S^2$ and $N$ a surface. In the next two subsections, we consider certain classes of surfaces in $\R^3$, which often arise in applications and are determined by augmented curves in $\R^4$. In the last two subsections, we consider certain spherical regions as well as hurricane tracks together with their intensities. In both cases, we utilize the Riemannian distance from subsection \ref{subsec:1.3} to $S^2\times \R$, which is a homogeneous space (recall that $S^2$ can be identified with $SO(3)/SO(2)$).

For our example applications, we present geodesic paths representing deformations, minimizing the elastic energy within the SRV framework. We remark, that in a Riemannian setting, distance and geodesics are essential Building blocks for many major issues in the morphology and shape analysis, such as computation of mean and test statistics as well as principal component or geodesic analysis. Moreover, besides statistical analysis, also some methods for clustering and classification use Riemannian metrics and geodesics. 

For the code implementing our approach, which particularly includes Riemannian optimization for the computation of geodesic paths, we utilized our publicly available python package \url{https://github.com/morphomatics}, introduced in \cite{morphomatics}.     
\subsection{Tubes}
A tube or canal surface $c$ is a one-parameter family of circles, whose centers constitute a regular curve $\ga$ such that the circles are perpendicular to $\ga$. More precisely, denoting the radii of the circles by $r$, 
\[
c(s, .)=\ga+r(N\cos s+B\sin s),\, 0\leq s\leq 2\pi,
\]
where $N$ and $B$ are the normal and binormal of the curve $\ga=\ga (t),\,t\in\Dx$, resp. Due to the unique correspondence of $c$ to $(\ga, r)$, comparison of tubes reduces to comparison of curves in $\R^4$. Figure \ref{fig:tube} shows some examples of shortest paths of tubes.  Real world applications include a variety of fields such as examination of vein, pipes, capsules and plant roots. Clearly, tubes include surfaces of revolution.
\begin{figure}[h]
    \vspace{-2ex}
    \centering
    \includegraphics[width=.8\linewidth]{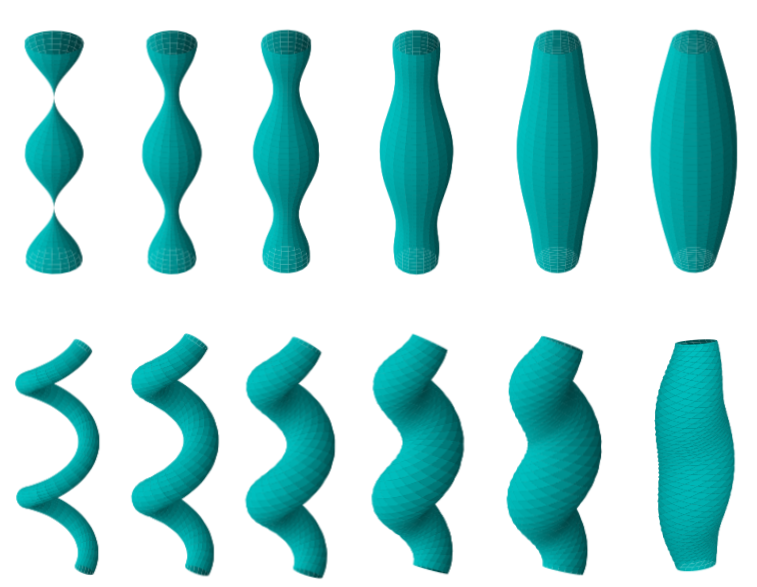}
    \caption{Two shortest paths of tubes}
    \label{fig:tube}
\end{figure}
\subsection{Ruled Surfaces}
A ruled surface is formed by moving a straight line segment (possibly with varying length) along a base curve. More precisely, let $\ga$ be a curve in $\R^3$ and $v$ a unit vector field along $\ga$. Then 
\[
c(s,.)=\ga +sv,\,s\in I,
\] 
parametrizes a ruled surface generated by $(\ga,v)$. Figure \ref{fig:ruled} depicts an example, where each surface consists of straight line segments connecting the blue (for better visibility) curves $\ga$ and $\ga+v$. The class of ruled surfaces includes many prominent surfaces such as cone, cylinder, helicoid (a minimal surface) and M\" obius strip. They arise in manufacturing (construction by bending a flat sheet), cartography, architecture and biochemistry (secondary and tertiary structure of protein molecules). 
\begin{figure}[htb]
    \centering
    \includegraphics[width=.8\linewidth]{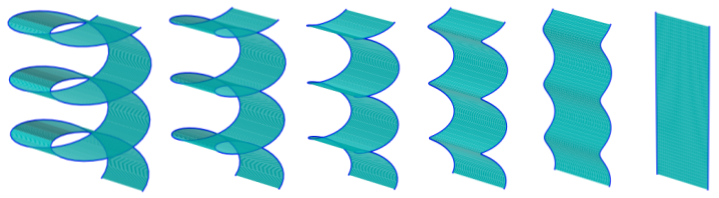}
    \caption{Shortest path of ruled surfaces}
    \label{fig:ruled}
\end{figure}
\subsection{Spherical Strips}
Let $\exp$ denote the exponential map of the unit two-sphere $S^2$. We recall that for any non-zero tangent vector to $S^2$ at a point $x$:
\[
\exp_x(v)=\cos(|v|)x+\sin(|v|)\frac{v}{|v|}
\]
and $\exp_x(0)=x$. Now, let $\ga$ be a curve in $S^2$ with binormal $B$ (cross product of $\ga$ and its unit tangent), and $r$ a scalar function along $\ga$. Then, the map $c$ given by
\[
c(s,.):=\exp_\ga s(rB),\, s\in I,
\]
parametrizes a spherical strip with bandwidth $r$. Figure \ref{fig:strip} depicts an example of the shortest path between two spherical curves comprised with their bandwidth functions visualised as strips.
\begin{figure}[h]
    \centering
    \includegraphics[width=1\linewidth]{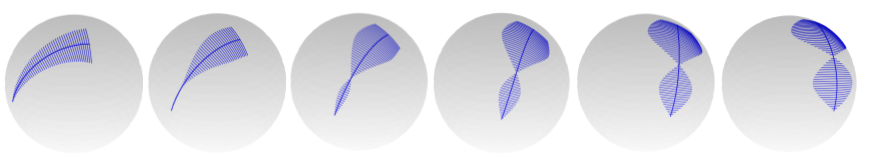}
    \vspace{-4ex}
    \caption{Shortest path of spherical strips}
    \label{fig:strip}
\end{figure}
\vspace{-4ex}
\subsection{Hurricane Tracks}
Hurricanes belong to the most extreme natural phenomena and can cause major impacts regarding environment, economy, etc. Intensity of a hurricane is determined by the maximum sustained wind (maxwind), monotonically classifying the storms into categories (due to Saffir–Simpson wind scale; for instance, maxwind $\geq 137$ knots corresponds to category 5). 
Due to their major impacts on economy, human life and environment, as well as extreme variability and complexity, hurricanes have been studies in a large number of works. For our example, we used the HURDAT 2 database provided by the U.S. National Oceanic and Atmospheric Administration publicly available on \url{https://www.nhc.noaa.gov/data/}, supplying latitude, longitude, and maxwind on a 6 hours base of Atlantic hurricanes.

\begin{figure}[htb]
    \centering
    \includegraphics[width=.96\linewidth]{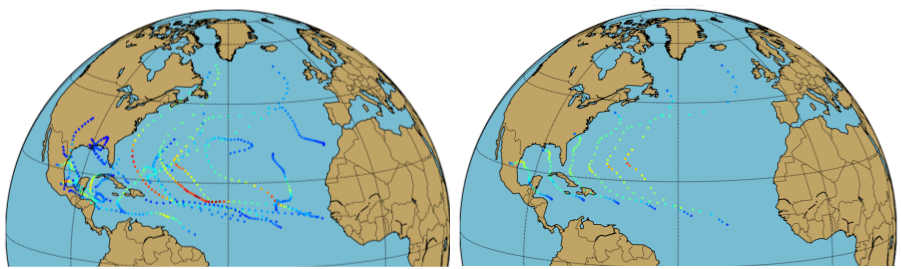}
    \includegraphics[width=.028\linewidth]{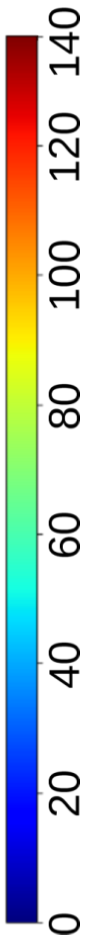}
    \vspace{-2ex}
    \caption{2010 Atlantic hurricane tracks (left) and the shortest path between two of them (right) with color-coded maximum sustained wind (in knots)}
    \label{fig:hur}
\end{figure}
We represent the tracks as discrete trajectories in $S^2$. For further details and comparison with other approaches, we refer to \cite{sri2016functional,homo} and the recent work~\cite{nava2023hur}. The latter, also provides statistical analysis and a classification of hurricane tracks in terms of their intensities. Fig.~\ref{fig:hur} illustrates this data set with a visualization of the 2010 hurricane tracks and a shortest path, where the intensities, considered as auxiliary information, are color-marked.
\section{Conclusion}
In this paper, we analysed the behaviour of speed and curvature under the square root velocity framework for elastic approach to plane curves. Moreover, we applied an extension of this framework to homogeneous Spaces, to metrically compare augmented curves and special surfaces, generated by those curves, using a natural construction via the Riemannian exponential map. Our approach, allows for computationally efficient determination of geodesic paths in the shape spaces of the respective classes of surfaces. Our example applications include tubes, ruled surfaces, spherical strips and hurricane tracks. Future work includes further real world applications, particularly concerning statistical analysis of longitudinal data such as comparison of group wise trends within a hierarchical model as well as classification and prediction. 
\label{sec:3}
\subsubsection{Acknowledgements}
This work was supported through the German Research Foundation (DFG) via individual funding (project ID 499571814).
%
%
%
\bibliographystyle{splncs04}
\bibliography{lit}
\end{document}